\documentclass{article}
\usepackage[journal=JST, lang=american]{ems-journal}

\newtheorem{thm}{Theorem}[section]

\newtheorem{lem}[thm]{Lemma}
\newtheorem{prop}[thm]{Proposition}

\theoremstyle{definition}

\newtheorem{defn}[thm]{Definition}
\newtheorem{eg}[thm]{Example}
\newtheorem{hyp}[thm]{Hypothesis}

\newtheorem{nota}[thm]{Notation}

\theoremstyle{remark}

\newtheorem{remk}[thm]{Remark}

\newcommand{\Rbb}{ {\mathbb R}}

\numberwithin{equation}{section}

\begin{document}

\title{A hot spots theorem for the mixed eigenvalue problem with small Dirichlet region}
\titlemark{A hot spots theorem for the mixed eigenvalue problem with small Dirichlet region}



\emsauthor{1}{
	\givenname{Lawford}
	\surname{Hatcher}
	\mrid{1540805}
	\orcid{0009-0007-8750-1965}}{L.~Hatcher}

\Emsaffil{1}{
	\department{Department of Mathematics}
	\organisation{Indiana University}
	\rorid{01kg8sb98}
	\zip{47401}
	\city{Bloomington, IN}
	\country{USA}
	\affemail{lhhatche@iu.edu}}

\classification[58J50]{35P05}


\begin{abstract}
We prove that on convex domains, first mixed Laplace eigenfunctions
have no interior critical points if the Dirichlet region is connected and sufficiently
small. We also find two seemingly new estimates on the first mixed eigenvalue to give
explicit examples of when the Dirichlet region is sufficiently small.
\end{abstract}

\maketitle

\section{Introduction}\label{intro}
Let $\Omega\subseteq\Rbb^2$ be a bounded Lipschitz domain. Let $D\subseteq \overline{\Omega}$ (where $\overline{\Omega}$ is the closure of $\Omega$ in $\Rbb^2$) be a compact set such that $\Omega\setminus D$ is connected. Consider the eigenvalue problem 
\begin{equation}\label{mixedeqn}
    \begin{cases}
        -\Delta u=\lambda u\;\;&\text{in}\;\;\Omega\setminus D\\
        u\equiv 0\;\;&\text{in}\;\;D\\
        \partial_{\nu}u\equiv 0\;\;&\text{in}\;\;\partial\Omega\setminus D.
    \end{cases}
\end{equation}
When $D$ has non-empty interior relative to $\overline{\Omega}$, we interpret this problem as deleting the relative interior of $D$ from $\Omega$ and enforcing Dirichlet boundary conditions on $\partial D$. We will suppose that $D$ is chosen such that this eigenvalue problem admits a discrete, increasing sequence of non-negative eigenvalues $\{\lambda_n^D\}$ and a corresponding orthonormal basis of eigenfunctions $\{u_n^D\}$ for $L^2(\Omega\setminus\text{int}(D))$ that (weakly) solve (\ref{mixedeqn}). When we wish to emphasize the dependence on the domain, we may write $\lambda_n^D(\Omega):=\lambda_n^D$. Note that because we did not place strong restrictions on $D$, the (weak) eigenfunctions may not vanish pointwise on $D$ (see \cite{felli}). We will assume throughout the paper that $D$ is chosen such that eigenfunctions of (\ref{mixedeqn}) vanish everywhere on $D$. For example, this holds if $D$ is the closure of a relatively open subset of $\partial\Omega$ or if $D\subseteq \Omega$ and $\Omega\setminus D$ is a Lipschitz domain. However, one can show that $D$ may be a somewhat more pathological set, such as a line segment in the interior of $\Omega$.\\
\indent For some subset $E\subseteq \Rbb^2$, define the \textit{diameter} of $E$ to be the quantity \[\sup\{|x-y|:x,y\in E\]
\begin{thm}\label{mainthm}
    Suppose that $\Omega$ is convex. There exists $\epsilon>0$ such that if $D$ is connected with diameter at most $\epsilon$, then $u_1^D$ has no critical points in $\Omega\setminus D$.\footnote{Recall that a domain is, by definition, open, so the theorem does not prohibit critical points in $\partial\Omega$.} Moreover, if $d$ denotes the diameter of $\Omega$ and $j_0\approx 2.40483$ is the first zero of the Bessel function of order zero, then it suffices to take \[\epsilon=\sqrt{\frac{1}{\pi}|\Omega|}\cdot \exp\Big(-\frac{4\pi}{j_0^2}\cdot\frac{d^2}{|\Omega|}\Big).\]
\end{thm}
\begin{remk}
    If $D$ is compactly contained in $\Omega$, then the hypothesis that $D$ be connected is necessary. We demonstrate this in Example \ref{connectedisnec} below. If $D$ is contained in $\partial\Omega$, then it is unclear to us whether connectivity is a necessary condition or if its requirement is merely an artifact of the proof.
\end{remk}
\begin{remk}
    Our proofs of Theorems \ref{mainthm}, \ref{miyamotoestimate}, and \ref{annulusoptimal} are heavily inspired by the proof of Miyamoto's Theorem A in \cite{miyamoto}. By extension, Theorems \ref{miyamotoestimate} and \ref{annulusoptimal} are modifications of Weinberger's proof of the maximality of the second Neumann eigenvalue of a Euclidean ball \cite{wein}. 
\end{remk}

The well-known hot spots conjecture of J. Rauch (see \cite{rauch}) states that extrema of solutions to the heat equation (with generic initial conditions) on a free membrane tend toward the boundary of the membrane as time tends toward infinity. It is not hard to see (see \cite{banuelosburdzy}) that this statement is equivalent to second Neumann eigenfunctions of the Laplace operator (i.e. second eigenfunctions of (\ref{mixedeqn}) with $D=\emptyset$) having extrema only on the boundary. A natural generalization of this problem is to ask for which sets $D$ does the heat equation with perfect refrigeration on $D$ and perfect insulation on $\partial\Omega\setminus D$ have solutions with extrema tending toward the boundary over time. Theorem \ref{mainthm} shows that if $\Omega$ is convex and $D$ is connected and sufficiently small, then this is the case.\\
\indent There has been recent interest by researchers on variants of the hot spots conjecture for the mixed problem. In 2024, the following pre-prints appeared: \cite{me}, \cite{LY}, \cite{rohleder}, and \cite{me2}. To our knowledge, these are the first papers on the topic since a handful of results in the early 2000s: \cite{P}, \cite{BP}, and \cite{BPP}.\\
\indent Theorem \ref{mainthm} is the combination of Propositions \ref{hotspots} and \ref{eigenestimate} below.
\begin{defn}
    Let $\Omega\subseteq \Rbb^2$ be a simply connected, bounded Lipschitz domain. Let $D\subseteq\overline{\Omega}$ be as above. For each $y$ in the smooth part of $\partial\Omega$, let $\nu(y)$ be the outward normal vector to the boundary at $y$. We say that the pair $(\Omega,D)$ is \textit{Neumann convex} if for any $x\in \Omega\setminus D$ and for any $y\in \partial\Omega\setminus D$ for which $\nu(y)$ is defined, the dot product $(y-x)\cdot\nu(y)$ is non-negative. Equivalently, $(\Omega,D)$ is Neumann convex if and only if $\partial \Omega\setminus D$ is contained in the boundary of the convex hull of $\Omega\setminus D$.
\end{defn}
For instance, if $\Omega$ is convex, then $(\Omega,D)$ is Neumann convex for any choice of $D\subseteq\overline{\Omega}$. The converse, however, is not true. For example, if $\Omega$ is a nodal domain of a second Neumann eigenfunction of a convex domain $\tilde{\Omega}$ (which one often expects to be non-convex) and $D$ equals the nodal set of the eigenfunction, then $(\Omega,D)$ is Neumann convex. For other examples of non-convex but Neumann convex domains, see Examples \ref{connectedisnec}, \ref{connectedisnec2}, and \ref{wild} in Section \ref{examples} below.
\begin{prop}\label{hotspots}
    Suppose that $(\Omega,D)$ is Neumann convex and that $D$ is connected. Let $d$ be the diameter of $\Omega$, and let $j_0\approx 2.4048$ be the first zero of the Bessel function $J_0(x)$. If \[\lambda_1^D\leq \Big(\frac{j_0}{d}\Big)^2,\] then $u_1^D$ has no critical points in $\Omega\setminus D$.
\end{prop}
In Section \ref{limits}, we solve two shape optimization questions for the mixed problem that appear to be new. Using these estimates, we show some explicit examples where this inequality holds, and we use these results to prove some new results on the hot spots conjecture. Roughly speaking, Theorem \ref{miyamotoestimate} applies when $\Omega$ is long and narrow with $D$ contained near an end of $\Omega$, while Proposition \ref{eigenestimate} applies when the diameter of $D$ is small.

\begin{thm}\label{miyamotoestimate}
    Let $(\Omega,D)$ be a Neumann convex pair embedded in the upper half plane $\{(x,y)\mid y\geq 0\}$ such that $D$ is contained in the closed second quadrant $\{(x,y)\mid x\leq 0, y\geq 0\}$. Let $\ell$ be the length of the projection of $\Omega$ onto the $y$-axis. Let $A$ be the area of the intersection of $\Omega$ with the first quadrant $\{(x,y)\mid x,y\geq 0\}$. The first mixed eigenvalue satisfies \[\lambda_1^D\leq \Big(\frac{\pi\ell}{2A}\Big)^2.\] This inequality is sharp, and equality is uniquely achieved when $\Omega$ is the rectangle $(0,A/\ell)\times (0,\ell)$ with $D=\{0\}\times [0,\ell]$. Moreover, if \[\frac{d\ell}{A}\leq \frac{2j_0}{\pi}\approx 1.531,\] then $\lambda_1^D\leq (j_0/d)^2$, so first mixed eigenfunctions have no critical points in $\Omega\setminus D$.
\end{thm}

Let $P_n$ denote a regular polygon with $n$ edges for $n\geq 3$. Let $e$ be the relative interior of an edge of $P_n$, and let $Q_n$ be the reflection of $P_n$ over $e$. Let $K_n=P_n\cup e\cup Q_n$. An application of Theorem \ref{miyamotoestimate} above and Theorem 1.1 of the author's earlier paper \cite{me} is the following, which we prove in Section \ref{examples}:
\begin{thm}\label{polygons}
    Second Neumann eigenfunctions of the Laplace operator on $K_n$ have no interior critical points. In particular, the hot spots conjecture holds for these domains. Moreover, the second Neumann eigenvalue of these domains is simple.
\end{thm}

We now give a brief outline of the paper. In Section \ref{limits}, we make estimates on the first mixed eigenvalues, showing in particular that this eigenvalue tends to zero with the diameter of the Dirichlet region. In Section \ref{proofs}, we provide proofs of Theorem \ref{mainthm} and Proposition \ref{hotspots}. To end the paper in Section \ref{examples}, we provide a number of examples and counterexamples illustrating our results. 

\section{First mixed eigenvalue estimates}\label{limits}
In this section, we prove an eigenvalue estimate that, together with Proposition \ref{hotspots}, proves Theorem \ref{mainthm}. We begin by recalling the variational characterization of the first mixed eigenvalue and eigenfunction. Let $H_D^1(\Omega)$ be the completion of the set of $C^{\infty}$ functions in $\overline{\Omega}$ that vanish on $D$ with respect to the norm \[\|u\|_{H_D^1}^2=\int_{\Omega}(|u|^2+|\nabla u|^2).\] Then we have

\begin{equation}\label{minmax}
    \lambda_1^D=\inf_{u\in H^1_D(\Omega)\setminus\{0\}}\frac{\int_{\Omega}|\nabla u|^2}{\int_{\Omega}|u|^2},
\end{equation}
and the minimizers of the functional above are exactly the first mixed eigenfunctions of the Laplacian. Constructing appropriate test functions for this functional allows us to give upper bounds on the first mixed eigenvalue. \\
\indent Let $B(0,R_1)$ denote the ball in $\Rbb^2$ centered at the origin with radius $R_1>0$. Consider the class of all pairs $(\Omega,D)$ as in Section \ref{intro} such that $D\subseteq B(0,R_1)$ and such that the area of $\Omega\setminus B(0,R_1)$ is some fixed positive number $V$. Theorem \ref{annulusoptimal} below shows that the annulus of inner radius $R_1$ and area $V$ maximizes the first mixed eigenvalue over all such pairs. This estimate is particularly useful when $\Omega$ is a doubly connected domain with $D$ equal to its inner boundary.\\
\indent The existing literature on eigenvalue optimization for the mixed problem is somewhat limited, but there do exist previous results in this area; see \cite{nehari}, \cite{payne}, \cite{hersch}, and \cite{bandle}. Hersch's result in \cite{hersch} solves a similar problem to Theorem \ref{annulusoptimal} in which the area of $\Omega$ itself as well as the length $L$ of the inner boundary are both fixed quantities. Using simple estimates on the eigenvalues of annuli (see Equation (\ref{annulusupper}) and Remark \ref{poin} below), Theorem \ref{annulusoptimal} gives a smaller upper bound on the first mixed eigenvalue than Hersch's theorem does when $R_1$ is small and $L$ is large.
\begin{nota}
    For $0<R_1<R_2$, let $A(R_1,R_2)=\{R_1<|(x,y)|<R_2\}$ be the annulus with inner radius $R_1$ and outer radius $R_2$ centered at the origin.
\end{nota}
\begin{hyp}\label{annsetup}
    Let $\Omega$ be a bounded Lipschitz domain, and let $D\subseteq \overline{\Omega}$ be as in Section \ref{intro}. We may translate $\Omega$ and $D$ such that $D$ is contained in the closed ball $\overline{B(0,R_1)}$, where $R_1$ equals the diameter of $D$. Let $\Omega_+=\Omega\setminus B(0,R_1)$, and let \[R_2=\sqrt{R_1^2+|\Omega_+|/\pi}.\]
\end{hyp}

\begin{thm}\label{annulusoptimal}
    Let $\Omega$, $R_1$, and $R_2$ be as in Hypothesis \ref{annsetup}. Then \[\lambda_1^D(\Omega)\leq \lambda_1^{\partial B(0,R_1)}(A(R_1,R_2)).\] Equality holds if and only if $\Omega=A(R_1,R_2)$ with $D=\partial B(0,R_1)$.
\end{thm}
\begin{proof}
    Let $\phi$ be a non-negative first mixed eigenfunction for $A(R_1,R_2)$ with $D=\partial B(0,R_1)$. Since the first mixed eigenvalue is simple and $\phi$ does not change signs, $\phi$ must be a radial function. We claim that $\phi$ is strictly monotonic in the radial direction. Indeed, if not, then the set of critical points of $\phi$ contains a circle of radius $R_3\in (R_1,R_2)$. The restriction of $\phi$ to the annulus $A(R_3,R_2)$ is a non-constant, nowhere vanishing Neumann eigenfunction of $A(R_3,R_2)$. Such a function cannot exist, so the claim must hold.\\
    \indent Extend $\phi$ to $\{|(x,y)|>R_2\}$ by setting it equal to its (constant) value on the outer boundary of $A(R_1,R_2)$. Extend $\phi$ to vanish in the disk $\{|(x,y)|<R_1\}$. Note that $A(R_1,R_2)$ has area $|\Omega_+|$, so \[|A(R_1,R_2)\setminus \Omega_+|=|\Omega_+\setminus A(R_1,R_2)|.\] Then the estimates 
    \[\int_{\Omega}|\nabla\phi|^2=\int_{\Omega\cap A(R_1,R_2)}|\nabla\phi|^2\leq \int_{A(R_1,R_2)}|\nabla\phi|^2\;\;\text{and}\]
    \begin{align*}
        \int_{\Omega}|\phi|^2&=\int_{\Omega\cap A(R_1,R_2)}|\phi|^2+\int_{\Omega\setminus A(R_1,R_2)}|\phi|^2\\&\geq \int_{\Omega\cap A(R_1,R_2)}|\phi|^2+\int_{A(R_1,R_2)\setminus\Omega}|\phi|^2\\
        &=\int_{A(R_1,R_2)}|\phi|^2
    \end{align*}
    combine to give \[\lambda_1^D(\Omega)\leq \frac{\int_{\Omega}|\nabla\phi|^2}{\int_{\Omega}|\phi|^2}\leq \frac{\int_{A(R_1,R_2)}|\nabla \phi|^2}{\int_{A(R_1,R_2)}|\phi|^2}=\lambda_1^{\partial B(0,R_1)}(A(R_1,R_2)).\] If each of the above inequalities is actually equality, then either $\Omega=A(R_1,R_2)$ or $\Omega$ equals the union of $A(R_1,R_2)$ with a subset of $B(0,R_1)$ with non-empty interior. In the latter case, unique continuation implies that the eigenvalue inequality is strict. 
 \end{proof}

\begin{prop}\label{eigenestimate}
    Let $\Omega$, $R_1$, and $R_2$ be as in Hypothesis \ref{annsetup}. Suppose that $R_1\leq e^{-2}\cdot R_2$. Then \[\lambda_1^D(\Omega)\leq \frac{4}{R_2^2\ln(R_2/R_1)}.\] In particular, for fixed $\Omega$, $\lambda_1^D(\Omega)$ approaches $0$ as the diameter of $D$ approaches $0$.
\end{prop}
\begin{proof}
    We begin by bounding $\lambda_1^{\partial B(0,\epsilon)}(A(\epsilon,1))$ from above when $\epsilon<e^{-2}$. Let $(r,\theta)$ denote polar coordinates on $\Rbb^2$. Let $\phi(r,\theta)=1-\ln(r)/\ln(\epsilon)$. Then $\phi$ vanishes on $\{r=\epsilon\}$, so its restriction to the annulus is a valid test function. We have \[\int_{A(\epsilon,1)}|\nabla \phi|^2=\frac{-2\pi}{\ln\epsilon}.\] Note that for each fixed $r$, $\phi$ is increasing in $\epsilon$. Moreover, $r\mapsto r\phi(r)^2$ is a convex function whose derivative at $r=1$ equals $1-\frac{2}{\ln(\epsilon)}$. Since $\epsilon<e^{-2}$, this implies that $r\phi^2>2r-1\geq 0$ for $r\in (1/2,1)$. Hence \[\int_{A(\epsilon,1)}\phi^2\geq 2\pi\int_{1/2}^1(2r-1)dr=\pi/2.\] We then have 
    \begin{equation*}
        \lambda_1^{\partial B(0,\epsilon)}(A(\epsilon,1))\leq \frac{-4}{\ln\epsilon}.
    \end{equation*}
    Scaling a domain by a positive value $C$ rescales its eigenvalues by $1/C^2$. Thus, the more general annulus $A(R_1,R_2)$ satisfies the inequality \begin{equation}\label{annulusupper}\lambda_1^{\partial B(0,R_1)}(A(R_1,R_2))=\frac{1}{R_2^2}\lambda_1^{\partial B(0,R_1/R_2)}(A(R_1/R_2,1))\leq \frac{4}{R_2^2\ln(R_2/R_1)}.\end{equation} Theorem \ref{annulusoptimal} then gives the desired estimate.
\end{proof}

\begin{remk}\label{poin}
    Using a Poincar\'e-type inequality, one can show that the first mixed eigenvalue of the annulus is bounded below by \[\lambda_1^{B(0,R_1)}(A(R_1,R_2))\geq \frac{2}{R_2^2\ln(R_2/R_1)}.\] Hence, the estimate on $\lambda_1^{B(0,R_1)}(A(R_1,R_2))$ obtained in the proof of Proposition \ref{eigenestimate} is sharp up to a constant multiple. This should be compared with the more general result in \cite{felli}, which is stated in terms of the Sobolev capacity of $D$.
\end{remk}

\section{Proofs of main results}\label{proofs}
By Proposition $\ref{eigenestimate}$, Theorem \ref{mainthm} follows from Proposition \ref{hotspots}, which we prove in this section. We begin by reviewing some preparatory results on nodal sets (i.e. zero-level sets) of Laplace eigenfunctions. \\
\indent Nodal sets of Laplace eigenfunctions have been studied for many years. It is well known that non-constant functions satisfying $-\Delta u=\lambda u$ on an open set have nodal sets that are unions of real-analytic arcs (see, e.g., \cite{cheng}) and, in particular, have no isolated points. If such $u$ vanishes at a critical point $p$, then there exists a disk neighborhood $U$ of $p$ such that $(U\cap u^{-1}(\{0\}))\setminus\{p\}$ is the union of at least four arcs. Moreover, $U$ can be chosen such that $u$ is positive in half of the connected components of $U\setminus u^{-1}(\{0\})$ and negative in the other components. The following lemma places an additional restriction on the nodal sets of eigenfunctions whose eigenvalue is $\lambda_1^D$.
\begin{lem}\label{noloops}
    Let $u$ be a pointwise solution of $-\Delta u=\lambda_1^Du$ ($u$ is not necessarily equal to $u_1^D$) in $\Omega$ that extends continuously to $\overline{\Omega}$. Identify $u$ with its extension. If $u^{-1}(\{0\})$ contains a loop intersecting $\Omega$, then the region bounded by this loop intersects $D$.
\end{lem}
\begin{proof}
    Recall the variational formulation (\ref{minmax}) for $\lambda_1^D$ and the fact that the only minimizers of the functional are scalar multiples of $u_1^D$. Suppose for contradiction that $u^{-1}(\{0\})$ does contain a loop that bounds a topological disk $B$ that does not intersect $D$. Then $\phi:=u\chi_B$ (where $\chi_B$ is the indicator function for $B$) is an element of $H^1_D(\Omega)$, and integration by parts gives
    \[\int_{\Omega}|\nabla \phi|^2=\lambda_1^D\int_{\Omega}|\phi|^2\] since $\phi=0$ on $\partial\Omega$. This implies that $\phi$ is a scalar multiple of $u_1^D$. Since the loop intersects $\Omega$, the set $\Omega\setminus B$ is non-empty and open, so we obtain a contradiction to unique continuation. 
\end{proof}
Let $J_0$ be the order zero Bessel function of the first kind (see, for instance, Chapter 10 of \cite{specialfcns}), and let $j_0\approx2.4048$ be the first positive zero of $J_0(x)$. Note that $J_0(0)=1$, so $J_0$ is positive on $[0,j_0)$. Note also that $J'_0(0)=0$ and that $J'_0(x)<0$ on $(0,j_0]$. Let $d$ be the diameter of $\Omega$, and suppose that $D$ is chosen such that $\lambda_1^D\leq \big(\frac{j_0}{d}\big)^2$.
\begin{proof}[Proof of Proposition \ref{hotspots}]
    Suppose toward a contradiction that $u_1^D$ has a critical point $p\in \Omega$. Since $u$ does not vanish anywhere in $\Omega\setminus D$, we may suppose without loss of generality that $u_1^D\geq 0$. By translating $\Omega$, we may suppose without loss of generality that $p$ equals the origin. Using polar coordinates, let $w(x,y)=J_0(\sqrt{\lambda_1^D}r)$. For $(x,y)\in \overline{\Omega}$, we have $r<d$ (with strict inequality since $p$ is an interior point), from which it follows that $w>0$ in $\overline{\Omega}$. By definition, $w$ satisfies $-\Delta w=\lambda_1^D w$ in $\Omega$. Define $f(x,y)=u_1^D(0)w(x,y)-u_1^D(x,y)$. Then $-\Delta f=\lambda_1^D f$. Moreover, $f$ is non-constant since $\lambda_1>0$ and $f>0$ on $D$. We will use $f$ to construct a test function that contradicts the variational formulation (\ref{minmax}) for $\lambda_1^D$.\\
     \indent We have $f(0)=0$ and $\nabla f(0)=0$ since $0$ is a critical point for both $u_1^D$ and $w$. Thus, there exists a neighborhood $U$ of $0$ such that $U\setminus f^{-1}(\{0\})$ has at least four connected components, with $f$ taking alternating signs on each connected component. By Lemma \ref{noloops} and using that $f>0$ on $D$, $\Omega\setminus f^{-1}(\{0\})$ has at least four connected components, and $f$ is positive in at least two of these components. Since $D$ is connected, there exists a unique connected component of $\Omega\setminus f^{-1}(\{0\})$ whose closure contains $D$. Thus, there exists a connected component $B$ of $\Omega\setminus f^{-1}(\{0\})$ on which $f$ is positive and whose closure does not intersect $D$. Let $\phi=f\chi_B\in H^1_D(\Omega)$.\\
     \indent We claim that $\partial_{\nu}\phi<0$ almost everywhere on $\partial\Omega\cap \partial B$. Indeed, since $|(x,y)|< d$ on $\partial\Omega$ and since $u_1^D$ is Neumann on $\partial\Omega\setminus D$, the following holds for almost every $(x,y)\in \partial\Omega\setminus D$: \[\partial_{\nu}\phi(x,y)=u_1^D(0)\partial_{\nu}w(x,y)=u_1^D(0)\sqrt{\lambda_1^D}J_0'\Big(\sqrt{\lambda_1^D}|(x,y)|\Big)\frac{(x,y)}{|(x,y)|}\cdot\nu.\] Because $(\Omega,D)$ is Neumann convex, we have $(x,y)\cdot\nu\geq 0$ for all $(x,y)\in\partial\Omega\setminus D$, so $\partial_{\nu}\phi\leq 0$ on $\partial\Omega\cap \partial B$. Using integration by parts and our knowledge of where $\phi$ vanishes, we have \[\int_{\Omega}|\nabla\phi|^2=\lambda_1^D\int_{\Omega}\phi^2+\int_{\partial\Omega\cap \partial B}\phi\partial_{\nu}\phi\leq \lambda_1^D\int_{\Omega}\phi^2.\] Dividing both sides of this inequality by $\int_{\Omega}\phi^2$ gives a contradiction to the variational characterization (\ref{minmax}) of $\lambda_1^D$ since $\phi$ vanishes on a non-empty open set.
\end{proof}

    \begin{proof}[Proof of Theorem \ref{mainthm}]
        The first statement follows from Propositions \ref{hotspots} and \ref{eigenestimate}. We now quantify which values of $\epsilon$ are sufficiently small for the first statement of the theorem to hold. Let $R_1$ and $R_2$ be as in Hypothesis \ref{annsetup}, and let $d$ denote the diameter of $\Omega$. Then the inequality $|\Omega_+|\geq |\Omega|-\pi R_1^2$ implies that $R_2\geq \sqrt{|\Omega|/\pi}$. Recall that the ratio $d^2/|\Omega|$ is minimized by a Euclidean disk, so $d^2/|\Omega|\geq 4/\pi$. Since $4\pi/j_0^2>\pi/2$, we get \[R_1\leq \sqrt{|\Omega|/\pi}\exp\Big(-\frac{4\pi}{j_0^2}\cdot \frac{d^2}{|\Omega|}\Big)\leq \sqrt{|\Omega|/\pi}\exp\Big(-\frac{\pi}{2}\cdot\frac{4}{\pi}\Big)=e^{-2}\sqrt{|\Omega|/\pi}\leq e^{-2}R_2.\] Thus, by Proposition \ref{eigenestimate}, we have 
        \begin{equation}\label{est1}
            \lambda_1^D(\Omega)\leq \frac{4}{R_2^2\ln(R_2/R_1)}.
        \end{equation} By hypothesis, we have \[R_1\leq \sqrt{|\Omega|/\pi}\cdot \exp\Big(-\frac{4\pi}{j_0^2}\cdot\frac{d^2}{|\Omega|}\Big)\leq R_2\exp\Big(-\frac{4\pi}{j_0^2}\cdot\frac{d^2}{|\Omega|}\Big),\] which may be rearranged to give \[\frac{4}{(|\Omega|/\pi)\ln(R_2/R_1)}\leq \Big(\frac{j_0}{d}\Big)^2.\] Concatenating this with Equation (\ref{est1}) gives \[\lambda_1^D(\Omega)\leq \frac{4}{(|\Omega|/\pi)\ln(R_2/R_1)}\leq \Big(\frac{j_0}{d}\Big)^2,\] and the result follows from Proposition \ref{hotspots}.
    \end{proof}

    We next prove Theorem \ref{miyamotoestimate}, which in some cases gives a more relaxed geometric condition than Theorem \ref{mainthm} under which a Neumann convex pair $(\Omega,D)$ has first mixed eigenfunctions with no interior critical points.

\begin{proof}[Proof of Theorem \ref{miyamotoestimate}]
    Orient $(\Omega,D)$ in the plane as in the paragraph preceding the theorem statement. Let $R$ denote the rectangle $(0,A/\ell)\times(0,\ell)$. Let $\Omega_+=\Omega\cap \{(x,y)\in\Rbb^2\mid x>0\}$. Then $R$ and $\Omega_+$ have the same area, and $|R\setminus\Omega_+|=|\Omega_+\setminus R|$. Define a test function \[\phi(x,y)=\begin{cases}
        0\;\;&\text{if}\;\;x\leq 0\\
        \sin\Big(\displaystyle\frac{\pi\ell x}{2 A}\Big)\;\;&\text{if}\;\;0<x<A/\ell\\
        1\;\;&\text{if}\;\;A/\ell\leq x.
    \end{cases}\]
    Then $\phi|_{\Omega}\in H^1_D(\Omega)$. The two estimates 
    \[\int_{\Omega}|\nabla \phi|^2=\int_{\Omega\cap R}|\nabla \phi|^2\leq \int_R|\nabla \phi|^2\;\;\text{and}\]
    \[\int_{\Omega}|\phi|^2=\int_{\Omega_+\cap R}|\phi|^2+\int_{\Omega_+\setminus R}|\phi|^2\geq \int_{\Omega_+\cap R}|\phi|^2+\int_{R\setminus \Omega_+}|\phi|^2=\int_R|\phi|^2\]
    combine to give \[\lambda_1^D\leq \frac{\int_{\Omega}|\nabla\phi|^2}{\int_{\Omega}|\phi|^2}\leq \frac{\int_R|\nabla\phi|^2}{\int_{R}|\phi|^2}=\Big(\frac{\pi\ell}{2A}\Big)^2.\] The sharpness of the estimate and uniqueness of the extremizer follow from a similar argument to that given in the proof of Theorem \ref{annulusoptimal}.\\
    \indent Now suppose that, $d\ell/A\leq 2j_0/\pi$. Then we have \[\lambda_1^D\leq \Big(\frac{\pi\ell}{2A}\Big)^2\leq \Big(\frac{j_0}{d}\Big)^2.\] The last statement follows from this estimate combined with Proposition \ref{hotspots}.
\end{proof}

\section{Examples}\label{examples}
In this section, we construct several examples that satisfy the hypotheses of Proposition \ref{hotspots} and Theorem \ref{mainthm} as well as examples illustrating the necessity of the connectivity hypothesis in these results. We begin with the latter. 
\begin{eg}\label{connectedisnec}
    Let $\Omega$ be the square given by the product of intervals $(-1,1)\times (-1,1)$. We first show that if we allow $D$ to have two connected components, then there exist arbitrarily small sets $D$ in $\Omega$ such that $u_1^D$ has an interior critical point. However, we do not know whether this critical point is a local extremum.\\
    \indent Let $0<\epsilon<\frac{1}{2}$. Let $D_{\epsilon}$ equal the union of the disks $B((\pm\epsilon,0),\epsilon/2)$. See Figure \ref{counterex}. Then the first eigenfunction $u_1^{D_{\epsilon}}$ is even about both the $x$- and $y$-axis since it is non-negative. Hence $\partial_xu_1^{D_{\epsilon}}\equiv 0$ on the $y$-axis and $\partial_yu_1^{D_{\epsilon}}\equiv 0$ on the $x$-axis. In particular, the origin is a critical point of $u$ for all such $\epsilon$.\\
    \begin{figure}
        \centering
        \includegraphics[width=0.6\linewidth]{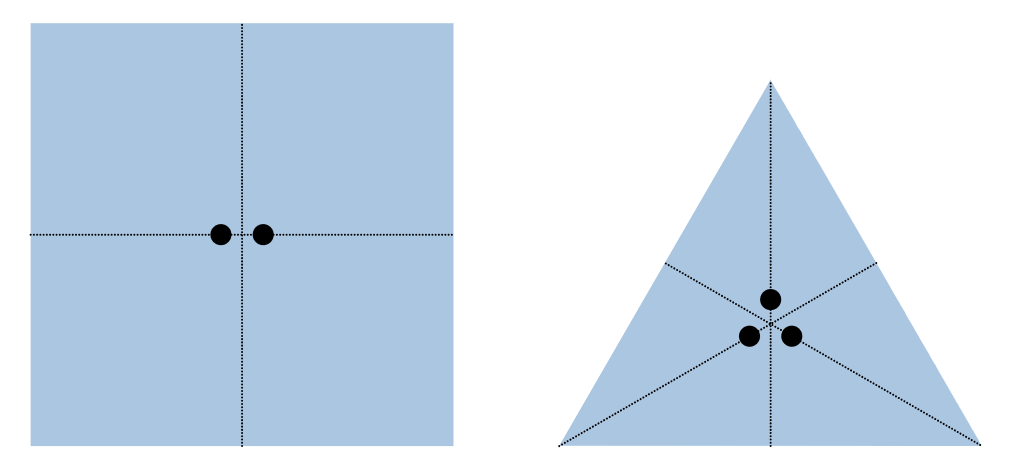}
        \caption{The pairs $(\Omega,D)$ constructed in Examples \ref{connectedisnec} and \ref{connectedisnec2}---each with $\epsilon=0.1$. The dotted lines show the lines of symmetry for each figure. These graphics were created with Desmos (2024).}
        \label{counterex}
    \end{figure}
\end{eg}
\begin{eg}\label{connectedisnec2}
    We now show that if $D$ is allowed to have three connected components, then there exists a convex domain with the diameter of $D$ arbitrarily small and such that the first mixed eigenfunction has an interior local extremum. Let $\Omega$ be an equilateral triangle, and let $\epsilon>0$. Number the vertices $v_1$, $v_2$, $v_3$. For each $i=1,2,3$, let $B_i$ denote the angle bisector at vertex $v_i$. Let $p_i$ be the point on $B_i$ between $v_i$ and the incenter of $\Omega$ such that $p_i$ is distance $2\epsilon/\sqrt{3}$ from the incenter. Let $D_{\epsilon}$ equal the union of the closed disks of radius $\epsilon/2$ centered at the $p_i$. For $\epsilon$ sufficiently small, $D_{\epsilon}$ is contained in $\Omega$ and has three connected components. See Figure \ref{counterex}. Since $D$ is invariant under the isometry group of $\Omega$ and a first mixed eigenfunction $u_1$ does not change signs in $\Omega$, it follows that $u_1$ is invariant under the isometry group of $\Omega$. Thus, $u_1$ restricts to a first mixed eigenfunction of each fundamental domain $\Omega'$ for this group action with Dirichlet conditions on $\Omega'\cap D_{\epsilon}$ and Neumann conditions on $\partial\Omega'\setminus D_{\epsilon}$. By Proposition 2.1 of \cite{erratum}, each acute vertex of $\Omega'$ is a local extremum of $u_1$. It follows that the incenter of $\Omega$ is a local extremum of $u_1$.
\end{eg}

\begin{eg}\label{disk}
    Let $\Omega$ denote the unit disk, and let $D$ be a subarc of $\partial\Omega$ of length $\alpha\leq \pi$. Then, using the notation of Theorem \ref{miyamotoestimate}, we have \[\frac{d\ell}{A}=\frac{4}{\pi-\alpha/2+\sin(\alpha/2)\cos(\alpha/2)}.\] Numerical computation shows that $d\ell/A\leq 2j_0/d$ as long as $\alpha\leq 1.976$. In particular, if $D$ is a subarc of $\partial\Omega$ of length at most $\pi/2$, then $u_1^D$ has no interior critical points.
\end{eg}

\begin{eg}\label{wild}
    The definition of Neumann convexity allows Theorem \ref{miyamotoestimate} to apply to quite wild domains. Let $1<a$, and let $R$ be the rectangle $(0,a-1)\times (0,1)$. Let $R_-$ be the closed square $[-1,0]\times [0,1]$. Let $\gamma$ be any simple curve contained in $R_-$ joining the points $(0,0)$ and $(0,1)$ that can be locally expressed as the graph of a Lipschitz function. Let $\Omega$ be the region bounded by the edges of $R$ not contained in the $y$-axis and by $\gamma$. See Figure \ref{wildfig}. Then \[\frac{d\ell}{A}\leq\frac{\sqrt{a^2+1}}{a-1}\to 1\;\;\text{as}\;\;a\to\infty.\] Thus, for $a$ sufficiently large, $D=\gamma$ defines an eigenvalue problem for which $u_1^D$ has no critical points. 
    \begin{figure}
        \centering
        \includegraphics[width=0.8\linewidth]{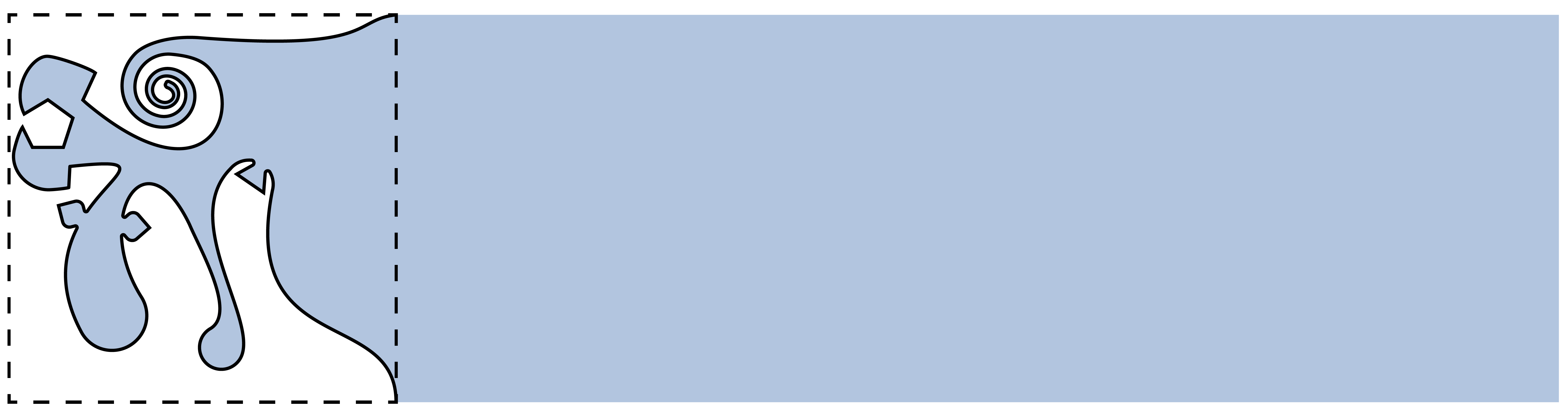}
        \caption{An example of a Neumann convex pair $(\Omega,D)$ discussed in Example \ref{wild}. The dashed line segments show the square $R_-$, and the solid black curve denotes $D=\gamma$. Created with Adobe Illustrator (2024).}
        \label{wildfig}
    \end{figure}
\end{eg}

\begin{thm}\label{regngons}
    Let $P_n$ denote a regular $n$-gon with $n\geq 4$, and let $D$ be a line segment contained in an edge of $P_n$ ($D$ may equal the entire edge). Then \[\lambda_1^D\leq \Big(\frac{j_0}{d_n}\Big)^2,\] where $d_n$ is the diameter of $P_n$. In particular, corresponding first mixed eigenfunctions have no interior critical points by Proposition \ref{hotspots}.
\end{thm}
\begin{proof}
    Normalize $P_n$ so that the distance from the center of $P_n$ to each vertex equals $1$. Suppose that $P_n$ is embedded in the first quadrant of $\Rbb^2$ with $D$ contained in the $y$-axis as in Theorem \ref{miyamotoestimate}. Using the notation of Theorem \ref{miyamotoestimate}, we have \[\frac{d\ell}{A}\leq \frac{4}{A}=\frac{4}{n\sin(\pi/n)\cos(\pi/n)}\] for all $n\geq 4$. The right-hand side of this inequality is decreasing in $n$. One can numerically check that we then have \[\frac{d\ell}{A}\leq \frac{4}{A}< \frac{2j_0}{\pi}\] for $n=7$ and therefore for all $n\geq 7$. However, the latter inequality fails to hold for $n\leq 6$, so the cases $n=4,5,6$ require other arguments.\\
    \indent One can check the $n=4$ case directly using the first mixed eigenfunction with $D$ equal to an edge of $P_4$ as a test function when $D$ is a subset of an edge.\\
    \indent For $n=5$ with $P_5$ normalized as above, we have \[\frac{d\ell}{A}=\frac{2\sin(2\pi/5)(1+\cos(\pi/5))}{5\sin(\pi/5)\cos(\pi/5)}< \frac{2j_0}{\pi}\] by numerical computation.\\
    \indent The case $n=6$ is the most involved. We will show directly that $\lambda_1^D\leq (j_0/d)^2$ without reference to Theorem \ref{miyamotoestimate}. Rescale, translate, and rotate $P_6$ such that $D$ is contained in the $y$-axis, $P_6$ has a vertex at the origin, and such that the length $\ell$ of the projection of $P_n$ onto the $x$-axis equals $1$. So $P_6$ has diameter $2/\sqrt{3}$. Let $R$ be the rectangle $(0,1)\times (0,1/\sqrt{3})$. Let $\tilde{R}$ be the rectangle $(0,1/2)\times (0,1/2\sqrt{3})$. Note that $R\subseteq P_6$ (perhaps after reflecting $P_6$ about the $y$-axis) and that $P_6\setminus R$ has twice the area of $\tilde{R}$. Define a test function $\phi(x,y)=\sin(\pi x/2)\in H^1_D(P_6)$. Then using the monotonicity of $\phi|_{P_6}$ and its derivative in $x$, we have 
    \[\int_{P_6}|\nabla\phi|^2\leq \int_R|\nabla\phi|^2+2\int_{\tilde{R}}|\nabla\phi|^2=\frac{1}{2\sqrt{3}}\Big(\frac{\pi}{2}\Big)^2\Big(\frac{3}{2}+\frac{1}{\pi}\Big)\;\;\text{and}\]
    \[\int_{P_6}\phi^2\geq \int_R\phi^2+2\int_{\tilde{R}}\phi^2=\frac{1}{2\sqrt{3}}\Big(\frac{3}{2}-\frac{1}{\pi}\Big),\] so we have \[\lambda_1^D\leq \frac{\int_{P_6}|\nabla\phi|^2}{\int_{P_6}\phi^2}\leq \Big(\frac{\pi}{2}\Big)^2\cdot\frac{\frac{3}{2}+\frac{1}{\pi}}{\frac{3}{2}-\frac{1}{\pi}}< \Big(\frac{j_0}{d}\Big)^2,\] where the last inequality holds by direct numerical computation. 
\end{proof}

\begin{proof}[Proof of Theorem \ref{polygons}]
Let $u$ denote a second Neumann eigenfunction of $K_n$. Let $e$ be the line of reflection of $K_n$ as in the definition given in Section \ref{intro}, and suppose that $e$ is contained in the $x$-axis. Then $u$ decomposes into a sum of even and odd eigenfunctions about this line:
\[u_{\text{even}}(x,y)=\frac{1}{2}(u(x,y)+u(x,-y))\]
\[u_{\text{odd}}(x,y)=\frac{1}{2}(u(x,y)-u(x,-y)).\]
We claim that $u_{\text{even}}$ is identically equal to zero. If not, then it restricts to a second Neumann eigenfunction of $P_n$. In this case, every second Neumann eigenfunction of $P_n$ extends via reflection to a second Neumann eigenfunction of $K_n$. By Lemma 3.3 of \cite{jm}, the nodal set of $u_{\text{even}}$ cannot have distinct endpoints in $e$. Thus, the nodal set of $u_{\text{even}}$ has end points in two disjoint edges of $P_n$. By precomposing $u_{\text{even}}$ by an appropriate rotational symmetry, there exists another second Neumann eigenfunction $v$ of $P_n$ whose nodal set does not intersect $e$. Since the extension of $v$ to $K_n$ by reflection has three nodal domains, we get a contradiction to Courant's nodal domains theorem, and the claim holds.\\
\indent Thus, the second Neumann eigenspace of $K_n$ is spanned by extensions of first mixed eigenfunctions of $P_n$ with Dirichlet conditions on $e$. By the simplicity of the first mixed eigenvalue, the second Neumann eigenvalue of $K_n$ is simple. By the Hopf lemma, these eigenfunctions have no critical points in $e$. By Theorem \ref{regngons}, the theorem holds for $n\geq 4$. For $n=3$, the theorem is a special case of Theorem 1.1 of \cite{me} or of Corollary 1.6 of \cite{LY}.
\end{proof}

We end the paper by using Proposition \ref{hotspots} to prove a more general result pertaining to the hot spots conjecture:
\begin{thm}\label{nodaldomain}
    Let $\Omega$ be a convex domain, and let $u$ be a second Neumann eigenfunction of $\Omega$. Let $\Omega_+=\{x\in\Omega\mid u(x)>0\}$. If the diameter of $\Omega_+$ is at most half of the diameter of $\Omega$, then $u$ has no interior critical points in $\Omega_+$. The same statement applies to the set on which $u$ is negative. 
\end{thm}
\begin{remk}
    We do not know of any examples of a domain with a second Neumann eigenfunction satisfying the hypothesis of Theorem \ref{nodaldomain}. It would be interesting to find out whether or not such a pair exists. 
\end{remk}
Before proving Theorem \ref{nodaldomain}, we remind the reader of an upper bound on the second Neumann eigenvalue on convex domains that appears as Corollary 2.1 in \cite{banuelosburdzy} and as Theorem 1 of Kr\"oger's paper \cite{kroger}.
\begin{lem}\label{neumannest}
    Let $\Omega$ be a convex domain with diameter $d$. The second Neumann eigenvalue $\mu_2$ of the Laplacian on $\Omega$ is at most $4(j_0/d)^2$.
\end{lem}

\begin{proof}[Proof of Theorem \ref{nodaldomain}]
    Let $\Omega$, $u$, and $\Omega_+$ be as in the theorem statement. Let $\mu_2$ be the eigenvalue associated to $u$. Because the restriction of $u$ to $\Omega_+$ is a first mixed eigenfunction with Dirichlet conditions on $u^{-1}(\{0\})$, we have $\lambda_1^{u^{-1}(\{0\})}(\Omega_+)=\mu_2$. Let $d$ be the diameter of $\Omega$, and let $d_+$ be the diameter of $\Omega_+$. Since $d_+\leq\frac{1}{2}d$, we have 
    \[\lambda_1^{u^{-1}(\{0\})}=\mu_2\leq4\Big(\frac{j_0}{d}\Big)^2\leq 4\Big(\frac{j_0}{2d_+}\Big)^2=\Big(\frac{j_0}{d_+}\Big)^2\]
    by Lemma \ref{neumannest}. By the discussion in Section \ref{intro}, $(\Omega_+,u^{-1}(\{0\}))$ is Neumann convex, so we may apply Proposition \ref{hotspots} to see that $u|_{\Omega_+}$ has no critical points in $\Omega_+$. To get the second statement of the theorem, we simply apply the same argument to the eigenfunction $-u$.
\end{proof}

\begin{ack}
The author would like to thank Matthew Lowe for illustrating Figure \ref{wildfig}. He would also like to thank Mark Ashbaugh for providing several helpful references on the mixed shape optimization problem.
\end{ack}


\end{document}